\documentclass[12pt]{amsart}

\usepackage{amscd,amsthm,amsfonts,amssymb,amsmath,mathrsfs,supertabular}
\usepackage[colorlinks=true]{hyperref}
\usepackage{fullpage}
\usepackage{tikz-cd}
\usepackage[all]{xy}
\usepackage{appendix}

\newcommand{\HK}{(L^\times /L^{\times p})_{\N \equiv 1}}
\newcommand{\HKv}{(L_{v}^\times /L_{v}^{\times p})_{\N \equiv 1}}
\newcommand{\HO}{(\CO_L^\times /\CO_L^{\times p})_{\N \equiv 1}}
\newcommand{\HOv}{(\CO_{L_v}^\times /\CO_{L_v}^{\times p})_{\N \equiv 1}}

 \newcommand{\BF}{{\mathbb {F}}}

\newcommand{\BQ}{{\mathbb {Q}}}

 \newcommand{\BZ}{{\mathbb {Z}}}

\newcommand{\CO}{{\mathcal {O}}}

\newcommand{\CU}{{\mathcal {U}}}

 \newcommand{\RH}{{\mathrm {H}}}

 \newcommand{\fp}{{\mathfrak{p}}}

\newcommand{\Div}{{\mathrm{Div}}} \renewcommand{\div}{{\mathrm{div}}}
\newcommand{\End}{{\mathrm{End}}}

\newcommand{\Gal}{{\mathrm{Gal}}} 

\newcommand{\Hom}{{\mathrm{Hom}}}

\renewcommand{\Im}{{\mathrm{Im}}}

\newcommand{\Ker}{{\mathrm{Ker}}}

\renewcommand{\mod}{\ \mathrm{mod}\ }

\newcommand{\Sel}{{\mathrm{Sel}}}

\newcommand{\ov}{\overline}

\newcommand{\lra}{\longrightarrow}
\newcommand{\ra}{\rightarrow}

\newcommand{\N}{\mathrm{N}}

                \newcommand{\Res}{\mathrm{Res}}

\theoremstyle{plain}
\newtheorem{thm}{Theorem}[section] 
\newtheorem{lem}[thm]{Lemma}  \newtheorem{prop}[thm]{Proposition}
 \newtheorem{defn}[thm]{Definition}
\newtheorem{hyp}[thm]{Hypothesis}

\theoremstyle{remark} \newtheorem{remark}{Remark}[section]
\theoremstyle{remark} 
\theoremstyle{remark} \newtheorem{ex}{Example}

\numberwithin{equation}{section}

\begin{document}
	\title[]{Bounding Selmer Groups of Superelliptic Jacobians via Class Groups}
	\author{Pengfei Wang}
	\begin{abstract}
Let $K$ be a number field containing a primitive $p$-th root of unity $\zeta_p$. Let $f(x)\in K[x]$ be a monic integral polynomial, and let $f_0$ denote its radical. Let $C/K$ be the superelliptic curve defined by $y^p=f(x)$, and let $J$ be its Jacobian variety. The variety $J$ admits multiplication by $\zeta_p$ over $K$; in particular, the endomorphism induced by $\Pi:=1-\zeta_p$ gives an isogeny of $J$ over $K$. Let $\Sel_\Pi(J)$ denote the Selmer group associated to $\Pi$. Under suitable hypotheses, we obtain bounds for $\Sel_\Pi(J)$ in terms of the $p$-torsion subgroup of the class group of $L:=K[x]/(f_0)$. Several examples illustrating the results are also discussed.
\end{abstract}
	
	\address{School of Mathematical Sciences, Tongji University, Shanghai 200092}
	\email{2311745@tongji.edu.cn}
	
	\maketitle
	\tableofcontents
\section{Introduction}
It has been known for decades that the Selmer group of an abelian variety is related to a certain class group. Schaefer provided an effective $2$-descent method for the Jacobians of hyperelliptic curves \cite{RN8}; he later related the $2$-Selmer group of the Jacobian of the hyperelliptic curve $y^2=f(x)$ over a number field $K$ to the class group of the \'etale algebra $K[x]/(f)$ \cite{RN30}. Poonen and Schaefer \cite{PoonenSchaefer1997} extended  Schaefer's method \cite{RN8} to superelliptic curves. Schaefer also fully developed the descent method for the Jacobian of a curve using functions on the curve \cite{RN31}. For certain rational elliptic curves, Li obtained an exact relation between their $2$-Selmer groups and the $2$-class groups of the associated cubic extensions\cite{RN33}. For elliptic curves over number fields, Barrera Salazar, Pacetti, and Tornar{\'i}a \cite{RN27} obtained bounds for their $2$-Selmer groups in terms of the $2$-class groups of the associated cubic extensions; subsequently, they obtained analogous results for Jacobians of hyperelliptic curves \cite{RN34}. In \cite{Stoll1,Stoll2}, under suitable hypotheses, Stoll established Selmer-rank formulae for Jacobians of the hyperelliptic curves $y^2=x^\ell+A$ over the $\ell$-th cyclotomic fields $K$ in terms of class numbers of $K(\sqrt{A})$. Shu \cite{Shu} obtained bounds for the Selmer groups of Jacobians of Fermat curves in terms of the class groups of cyclotomic fields, and established an exact equality relating Selmer ranks and class numbers under certain conditions.

In this paper, we establish a bound on the Selmer groups of the Jacobians of superelliptic curves in terms of the class groups of their associated étale algebras.

Let \(p\) be an odd prime, and let \(K\) be a number field containing a primitive \(p\)-th root of unity \(\zeta_p\). Let \(\mathcal O_K\) denote the ring of integers of \(K\). Let \(f \in \mathcal O_K[x]\) be a monic \(p\)-th power-free polynomial with decomposition
\[
f = f_1^{m_1} \cdots f_r^{m_r},
\]
where \(1 \leq m_i \leq p-1\) and the \(f_i\) are distinct monic irreducible polynomials over \(K\). Set \(f_0 = f_1 \cdots f_r\), \(d = \deg f_0\) and \(d^\prime = \deg f\), and assume that \(p \nmid dd^\prime\). Let \(C/K\) be the smooth projective curve with affine model
\[
C : y^p = f(x),
\]
and let \(J = \operatorname{Jac}(C)\). By Riemann-Hurwitz, the genus of \(C\) is
$
g = \frac{(p-1)(d-1)}{2}.
$ The automorphism \((x,y) \mapsto (x, \zeta_p y)\) induces an action of \(\mathbb Z[\zeta_p]\) on \(J\). Put
\[
\Pi := 1 - \zeta_p \in \End_K(J),
\]
and let $J[\Pi]$ be its kernel. 

Let $\Sigma$ be the set of all places of $K$, and for each $v\in \Sigma$ denote $K_v$ to be the corresponding local field. Define the $\Pi$-Selmer group as
\[\Sel_\Pi(J):=\Ker\left(\RH^1(K,J[\Pi]))\ra\prod_{v\in \Sigma}\RH^1(K_v,J)\right).\]
Set 
\[L := K[x]/(f_0(x)) \simeq \prod_i L_i,\] where $L_i$ are number fields. Let \(\operatorname{Cl}(L) = \prod_i \operatorname{Cl}(L_i)\) be the class group of $L$, and define
\[
\operatorname{Cl}(L)[p]_{N \equiv 0}
:=
\Ker\!\left(
N_{L/K} : \operatorname{Cl}(L)[p]
\longrightarrow \operatorname{Cl}(K)[p]
\right).
\]

For each place \(v \mid p\), following the method of \cite{RN27,RN34}, an explicit contribution $W_v$ to the local Kummer image is constructed using formal-group parametrization (cf. Section~\ref{subsec:natural-local-contribution}). 

\begin{hyp}\label{hyp:local}
The following conditions hold.
\begin{enumerate}
\item $(p,dd^\prime)=1$.
\item For every finite place \(v\in \Sigma\), the polynomial \(f_0\) satisfies \((\dagger)_v\) (see Definition~\ref{dagger}).
\item For every place \(v\mid p\), after an integral translation, the coefficient of
\(x^{d-1}\) in \(f_0\) lies in the maximal ideal
\(\mathfrak p_v\) of \(\mathcal O_v\), where \(\mathcal O_v\) is the
ring of integers of \(K_v\).
\end{enumerate}
\end{hyp}

\begin{remark}
For item (2), it follows from Proposition~\ref{root condition} that if $v$ does not divide the discriminant of $f_0$, then $(\dag)_v$ holds.
\end{remark}
The main result is the following.

\begin{thm}\label{thm:intro-main}
Assume Hypothesis~\ref{hyp:local}. Then
\[
\begin{aligned}
\dim_{\mathbb F_p}\operatorname{Cl}(L)[p]_{N\equiv0}
-\sum_{v\mid p}
\left(r_v-1-\dim_{\mathbb F_p}W_v\right)
&\leq
\dim_{\mathbb F_p}\operatorname{Sel}_{\Pi}(J/K) \\
&\leq
\dim_{\mathbb F_p}\operatorname{Cl}(L)[p]_{N\equiv0}
+g[K:\mathbb Q(\zeta_p)].
\end{aligned}
\]
\end{thm}

\noindent\textbf{Layout.}
In Section~\ref{sec:pi-descent}, following Schaefer, we establish the framework for the $\Pi$-descent of Jacobians of superelliptic curves over any field containing $\zeta_p$ (cf. Proposition~\ref{Weil-Kummer}). In Section~\ref{sec:local-kummer-images}, for each local place $v$, we determine the local Kummer image (cf. Theorem~\ref{localcomputation}) and construct the explicit subspace $W_v$ of the local Kummer image via formal-group parametrization. In Section~\ref{sec:selmer-class-group}, with the local Kummer images obtained in Section~\ref{sec:local-kummer-images} at hand, we construct auxiliary Selmer structures to bound $\operatorname{Sel}_{\Pi}(J/K)$ (cf. Theorem~\ref{thm:selmer-classgroup-bound}). Finally, we present several explicit examples.

\noindent {\bf Acknowledgements.}
This work constitutes part of the author's doctoral dissertation, carried out under the supervision of Jie Shu. The author is deeply grateful to him for his guidance, encouragement, and numerous valuable suggestions.

\section{The $\Pi$-descent over general fields}
\label{sec:pi-descent}
Let \(F\) be a field of characteristic
\(0\) containing a primitive \(p\)-th root of unity
\(\zeta_p\), and let \(G_F=\operatorname{Gal}(\overline F/F)\) be the absolute
Galois group of \(F\). 
Let \(f\in F[x]\) be a monic \(p\)-th power-free polynomial, written as
$
f=\prod_{i=1}^r f_i^{m_i},
1\leq m_i\leq p-1,
$
where the \(f_i\) are distinct monic irreducible polynomials over \(F\).
Set
$
f_0=\prod_{i=1}^r f_i,
d=\deg f_0,
d'=\deg f=\sum_{i=1}^r m_i\deg f_i,
$
and assume that 
$
p\nmid dd^\prime.
$
The polynomial \(f_0\) determines the finite étale \(F\)-algebra
$
L=F[T]:=F[x]/(f_0(x)),
$
where \(T\) denotes the image of \(x\). After base change to
\(\overline F\), one has
$
L\otimes_F\overline F
\simeq
\overline F[x]/(f_0(x))
\simeq
\prod_{i=1}^d\overline F.
$

Let \(C/F\) be the smooth projective model of the affine curve
$
C: y^p=f(x),
$
and let \(J=\operatorname{Jac}(C)\) be its Jacobian. Since \(p\nmid d'\), there is a
unique point \(\infty\in C(F)\) above infinity. Since $C$ admits a $F$-rational point, the Jacobian $J(F)$ can be identified with the group of divisor classes of degree zero on $C$ defined over $F$ (cf.~\cite[Lemma~1]{RN481}); that is,
\[
J(F) \cong \Div^0(C)(F)/\sim,
\]
where $\sim$ denotes linear equivalence.

Let \(\alpha_1,\ldots,\alpha_d\) be the roots of \(f_0\) in
\(\overline F\) and write
$
R_i=(\alpha_i,0)
(1\leq i\leq d)
$
for the corresponding ramification points. Then we set $f(x)=\prod_{i=1}^d (x-\alpha_i)^{s_i}$. The covering
$
x:C\longrightarrow\mathbb P^1_F
$
is totally ramified at \(R_1,\ldots,R_d\) and, since \(p\nmid d'\), at
\(\infty\). The Riemann--Hurwitz formula gives
$
g(C)=\frac{(p-1)(d-1)}{2}.
$
The automorphism \((x,y)\mapsto (x,\zeta_p y)\) of \(C\) induces an automorphism of \(J\), which we identify with \(\zeta_p\) and regard
$
\mathbb Z[\zeta_p]\subset \operatorname{End}_F(J).
$
Put
\[
\Pi:=1-\zeta_p\in \operatorname{End}_F(J),
\]
and let $J[\Pi]$ be its kernel. We have the following proposition.
	\begin{prop}\label{prop:cohom}
		\begin{enumerate}
			\item The torsion group $J[\Pi](\ov{F})$ is an $\BF_p$-space of dimension $d-1$ with 
			$[R_1-\infty],[R_2-\infty],\cdots,[R_{d-1}-\infty]$ as a basis.
			\item As a finite group scheme, $J[\Pi]\cong \Res^1_{L/F}\mu_p$ over $F$.
			\item There is a canonical isomorphism
			\[
			\RH^1(F,J[\Pi]) \cong (L^\times/L^{\times p})_{\N \equiv 1},
			\]
			where $(L^\times/L^{\times p})_{\N \equiv 1}$ denotes the subgroup of
			classes whose norm lies in $F^{\times p}$.
		\end{enumerate}
	\end{prop}
	
	\begin{proof}
		See \cite[\S 6, \S 8]{PoonenSchaefer1997}.
        \end{proof}

	\subsection{The explicit $x-T$ map and Kummer map}Following Schaefer \cite{RN31}, we use functions on the curve, to describe explicitly
the descent map associated with the isogeny \(\Pi\). Recall that $L = F[x]/(f_0(x))$. Let $T$ be the image of $x$ in $L$. Then
	\begin{equation*}
		L \otimes_F \bar F = \bar F[T]/(f_0(T))
		\cong \prod_{i=1}^d \bar F[T]/(T-\alpha_i)
		\cong \prod_{i=1}^d \bar F,
	\end{equation*}
	where the isomorphism is induced by
	\[
	T \longmapsto (T,\ldots,T) \longmapsto (\alpha_1,\ldots,\alpha_d),
	\]
	and where $\Gal(\bar F/F)$ acts trivially on $T$.
	
	To each divisor in $\Div^0(C)(F)$ away from the points $\{R_i|i=1,2,...,d\}$ and $\infty$, we can associate an invariant as follows,
	\[\delta_F\left(\sum a_i P_i\right)=(x-T)\left(\sum a_i P_i\right)=\prod (x(P_i)-T)^{a_i}\mod L^{\times p}.\]
	
	By Weil reciprocity (see, e.g., \cite{RN482}), on the principal divisors, one has
	\[
	\delta_F(\div(g))=(x-T)\bigl(\div(g)\bigr)
	= \left(\left(\frac{g(R_1)}{g(\infty)}\right)^p,\cdots,\left(\frac{g(R_d)}{g(\infty)}\right)^p \right) \in L^{\times p}.
	\]
	The last equality follows from 
	$
	\div(x-\alpha_i)=p(R_i- \infty)
	$
	 for each $i$.
	
	Thus we have an induced homomorphism $\delta_F:J(F)\lra L^\times/L^{\times p}$ whose kernel turns out to be the subgroup generated by the classes $\Pi D$ where $D\in J(F)$ (see \cite{RN483}). Then we have an injective map
	\[\delta_F:J(F)/\Pi J(F)\hookrightarrow L^\times/L^{\times p}.\]
	
	Together with \cite[Appendix~A, Lemma~A.1]{Shu}, the following proposition tells how to calculate $\delta_F(D)$ for general divisor classes $D\in J(F)$.

	\begin{prop}\label{X-Tmap}
		Every element of $J(F)$ admits a representative given by a degree-zero divisor defined over $F$ whose support contains neither the point at infinity $\infty$ nor any point with $y$-coordinate equal to $0$. In particular, let
		\[
		D = {}^{\sigma_1}Q + \cdots + {}^{\sigma_m}Q - m\infty,
		\]
		where ${}^{\sigma_1}Q,\ldots,{}^{\sigma_m}Q$ are the Galois conjugates over $F$ of a point
		$Q\in \mathcal C(\bar F)$ with $y(Q)\neq 0$. Then
		\[
		(x-T)([D]) \equiv \prod_{i=1}^m \bigl(x({}^{\sigma_i}Q)-T\bigr)
		\pmod{L^{\times p}}.
		\]
		Now suppose that
		\[
		D = (\alpha_1,0)+\cdots+(\alpha_r,0)-r\infty,
		\]
		where the $\alpha_i$ are conjugate over $F$ (possibly after reordering) and $r<d$. Recall that $f(x)=\prod_{i=1}^d (x-\alpha_i)^{s_i}$. If $u_i$ is an integer such that $u_is_i\equiv 1 \pmod p$, then
		\[
		(x-T)([D]) \equiv
		\prod_{i=1}^r (\alpha_i-T)^{u_i}+
		\prod_{i=r+1}^d (\alpha_i-T)^{-u_i}
		\pmod{L^{\times p}}.
		\]
	\end{prop}
	
	\begin{proof}
		Although \cite[Proposition 3.3]{RN31} treats curves defined by $y^p=f(x)$ with $f(x)$ has distinct roots. But the method of {\em loc. cit.} works for $f(x)$ with multiple roots. The proof of the first part goes verbatim as the first half of the proof of \cite[Proposition 3.3]{RN31} by taking $f(x)=\prod_{i=1}^d (x-\alpha_i)^{s_i}$ and $T=(\alpha_1,\alpha_2,...\alpha_d)$. As for the second assertion, one may argue analogously as in the later part of the proof of \cite[Proposition 3.3]{RN31} by taking $g=y-\prod_{i=1}^{r}(x-\alpha_i)$.
	\end{proof}
	
	Recall that the dual map of the isogeny $\Pi:J\ra J$ is given by the complex conjugate $\ov{\Pi}=-\Pi$, and $J[\Pi]=J[\ov{\Pi}]$. Let $e_\Pi:J[\Pi]\times J[\ov{\Pi}]\ra \mu_p$ be the Weil pairing induced by $\Pi$. 
	For each root \(\alpha_i\) of \(f_0\), put
	$
	R_i=[(\alpha_i,0)-\infty]\in J[\overline{\Pi}].
	$
	The Weil pairing associated with
	\(\Pi\) gives a \(G_F\)-equivariant homomorphism
	\[
	w:J[\Pi]\longrightarrow
	\mu_p(L\otimes_F\overline F),
	\qquad
	P\longmapsto
	\bigl(e_\Pi(P,R_i)\bigr)_{i=1}^d.
	\]
	The map $w$ is injective from the non-degeneracy of the Weil pairing. Passing to Galois cohomology, we obtain
	\[
	e:H^1(F,J[\Pi])
	\hookrightarrow
	H^1\!\left(F,\mu_p(L\otimes_F\overline F)\right),
	\]
	given on cocycles by
	\[
	e(\xi)
	=
	\left[
	\sigma\longmapsto
	\bigl(e_\Pi(\xi(\sigma),R_i)\bigr)_{i=1}^d
	\right].
	\]
	Under Kummer theory, the cohomology class represented by
	\[
	\sigma\longmapsto
	\frac{\sigma(\sqrt[p]{a})}{\sqrt[p]{a}}
	=
	\left(\sqrt[p]{a}\right)^{\sigma-1}
	\]
	is sent to the class of \(a\in L^\times/L^{\times p}\). Denote this
	Kummer isomorphism by
	\[
	i:
	H^1\!\left(F,\mu_p(L\otimes_F\overline F)\right)
	\xrightarrow{\sim} L^\times/L^{\times p}.
	\]
	Then we have
	\begin{equation*}
		i\circ e:
		H^1(F,J[\Pi])
		\hookrightarrow
		L^\times/L^{\times p}.
	\end{equation*}
	
	\begin{prop}\label{Weil-Kummer}
		The maps \(\delta_F\) and
		$
		i\circ e\circ\kappa_{F,\Pi}
		$
		coincide as maps from \(J(F)/\Pi J(F)\) to
		\(L^\times/L^{\times p}\). More precisely, the diagram
		\[
		\xymatrix{
			J(F)/\Pi J(F)
			\ar[r]^{\kappa_{F,\Pi}}
			\ar[dr]_{\delta_F}
			&
			H^1(F,J[\Pi])
			\ar[d]^{\,i\circ e}
			\\
			&
			L^\times/L^{\times p}
		}
		\]
		is commutative.
	\end{prop}
	
	\begin{proof}
		Taking $\phi:A\ra J$ to be $\Pi:J\ra J$ and applying  \cite[Theorem 2.3]{RN31}, it follows directly that $\delta_F=\kappa_{F,\Pi}$. 
	\end{proof}

	\begin{lem}\label{dimTate}
		For a local field $F$, we have
		\[\dim_{\BF_p} \Im(\kappa_{F,\Pi})=\frac{\dim_{\BF_p}\RH^1(F,J[\Pi])}{2}.\]
	\end{lem}
	\begin{proof}
		The Weil pairing $e_\Pi$ induces the local Tate duality
		\begin{equation}
			\label{pairing2}\RH^1(F,J[\Pi])\times \RH^1(F,J[\ov{\Pi}])\lra \BQ/\BZ.
		\end{equation}
		It is well-known from the local Tate dualities for finite groups and abelian varieties \cite{MilneADT} that the Kummer images $\Im(\kappa_{F,\Pi})$  and $\Im(\kappa_{F,\ov{\Pi}})$ are orthogonal complements of each other under the pairing (\ref{pairing2}).  The proposition follows by noting $\kappa_{F,\ov{\Pi}}=-\kappa_{F,\Pi}$.
	\end{proof}
\section{Local Kummer images}
\label{sec:local-kummer-images}
Let \(K\) be a number field containing a primitive \(p\)-th root of unity \(\zeta_p\). Let \(\Sigma\) denote the set of places of \(K\), and for each \(v \in \Sigma\), let \(K_v\) denote the corresponding local field. If \(v\) is finite, let \(\CO_v\) be the ring of integers of \(K_v\).

Since $K$ is totally imaginary, at an archimedean place, the relevant local cohomology group is trivial. We therefore restrict our attention to finite places.

\subsection{Local Kummer images}
\label{subsec:local-kummer-images}
Let $v$ be a finite place of \(K\),  let 
	\[f_0=f_{v,1}\cdots f_{v,r_v}\]
	denote the factorization of $f_0$ into product of irreducible polynomials over $K_v$ Accordingly, put
\[
L_v:=L\otimes_K K_v
\simeq \prod_{i=1}^{r_v}L_{v,i},
\]
where $L_{v,i}=K_v[x]/(f_{v,i})$. Let \(\mathcal O_{L_v}\) denote the ring of inegers of
\(L_v\), and set
\[
d_v=
\begin{cases}
	[K_v:\mathbb Q_p(\zeta_p)],& v\mid p,\\
	0,& v\nmid p.
\end{cases}
\]
By Proposition~\ref{prop:cohom}(3) and Proposition~\ref{Weil-Kummer}, we identify the local Kummer map $\kappa_{K_v,\Pi}$ with $\delta_{K_v}$, and henceforth denote it simply by:
\[
\delta_v\colon
J(K_v)/\Pi J(K_v)
\longrightarrow
\HKv.
\]
	\begin{defn}\label{dagger}
	For a finite place $v\in \Sigma$, we say that \(f_0\) satisfies condition \((\dagger)_v\) at \(v\) if the natural
	homomorphism
	\[
	\mathcal O_v[x]/(f_0(x))
	\longrightarrow
	\prod_{i=1}^{r_v}\mathcal O_v[x]/(f_{v,i}(x))
	\]
	is an isomorphism.
\end{defn}
\begin{prop}\label{root condition}
	For a finite place $v\in \Sigma$, the following conditions are equivalent:
	\begin{enumerate}
		\item The polynomial \(f_0\) satisfies \((\dagger)_v\).
		\item For maximal ideal $\mathfrak p_v\subset \mathcal O_{\overline {K_v}}$, every pair \(i\neq j\), every root \(\alpha\) of \(f_{v,i}\), and
		every root \(\beta\) of \(f_{v,j}\), one has $ \fp_v\nmid \alpha-\beta $
	\end{enumerate}
\end{prop}

\begin{proof}
	By induction, it suffices to treat the case of two factors. Suppose
	$f_0 = f_1 f_2$. Consider
	\[
	\phi : \mathcal O_v[x]/(f_0) \longrightarrow \mathcal O_v[x]/(f_1) \times \mathcal O_v[x]/(f_2).
	\]
	Since $f_1$ and $f_2$ are coprime in $K_v[x]$, the map $\phi$ is injective. Hence the isomorphism condition reduces to proving that $\phi$ is surjective.
	
	Assume first that $\phi$ is surjective. Then $(1,0)$ lies in its image, so there exists $b(x) \in \mathcal O_v[x]$ such that
	\[
	b(x) f_2(x) \equiv 1 \pmod{f_1}.
	\]
	If $\alpha$ is a root of $f_1$, then $b(\alpha) f_2(\alpha) = 1$. Now let $\beta$ be a root of $f_2$. If $\alpha - \beta \in \mathfrak p_v \mathcal O_{\overline{K_v}}$ for some maximal ideal $\mathfrak p_v \subset \mathcal O_{\overline{K_v}}$, then $\alpha \equiv \beta \pmod{\mathfrak p_v}$, whence $f_2(\alpha) \equiv f_2(\beta) = 0 \pmod{\mathfrak p_v}$. Thus $b(\alpha) f_2(\alpha) \in \mathfrak p_v$, contradicting $b(\alpha) f_2(\alpha) = 1$. This proves $(1) \Rightarrow (2)$.
	
	Conversely, assume (2). Then every $\alpha - \beta$ is a unit in $\mathcal O_{\overline{K_v}}$, so
	\[
	\operatorname{Res}(f_1, f_2) = \prod_{\alpha, \beta} (\alpha - \beta) \in \mathcal O_v^\times.
	\]
	Hence there exist $u(x), v(x) \in \mathcal O_v[x]$ such that
	\[
	f_1 u + f_2 v = \operatorname{Res}(f_1, f_2).
	\]
	Dividing by $\operatorname{Res}(f_1, f_2)$ yields
	\[
	f_1 u_1 + f_2 v_1 = 1
	\]
	with $u_1, v_1 \in \mathcal O_v[x]$. Thus $(1,0)$ and $(0,1)$ lie in the image of $\phi$, so $\phi$ is surjective.
\end{proof}

This subsection is devoted to the proof of the following.
\begin{thm}\label{localcomputation}
	Let \(v\) be a finite place of \(K\) and assume that Hypothesis~\ref{hyp:local} holds.
	
	\begin{enumerate}
		\item If \(v\nmid p\), then
		\[
		\Im(\delta_v)
		=
		\left(
		\mathcal O_{L_v}^{\times}/
		\mathcal O_{L_v}^{\times p}
		\right)_{\N\equiv 1}.
		\]
		
		\item If \(v\mid p\), then
		\[
		\left[\left(
		\mathcal O_{L_v}^{\times}/
		\mathcal O_{L_v}^{\times p}
		\right)_{\N\equiv 1}:
		\Im(\delta_v)\right]=p^{gd_v}
		\]
		
	\end{enumerate}
\end{thm}

\begin{proof}
    This theorem follows immediately from Proposition~\ref{thm:integrality} and
    Proposition~\ref{dimension}.
\end{proof}

First we prove the local Kummer image is contained in the unit group, and we record the following lemma.

\begin{lem}\label{lemma:negvaluation}
	Let $v:\overline{K_v}\ra\BQ$ be the valuation normalized so that $v(K_v^\times)=\BZ$. Let $P=(a,b)\in C(\overline{K_v})$ and suppose that $v(a)<0$.    
	Consider the divisor $D=\sum_{\sigma} (\sigma(P)-\infty)\in J(K_v)$
	where the sum is over the different conjugates of $P$. Then
	$\delta_v(D)=1\in \HO$.
\end{lem}

\begin{proof}
	The argument is analogous to that of Barrera Salazar, Pacetti,
	and Tornar\'ia \cite[Lemma~3.5]{RN34}. 
\end{proof}
\begin{prop}\label{thm:integrality}
	Under the hypotheses of Theorem~\ref{localcomputation}, we have
	\[
	\Im(\delta_v)
	\subset
	\left(
	\mathcal O_{L_v}^{\times}/
	\mathcal O_{L_v}^{\times p}
	\right)_{\N\equiv 1}.
	\]
\end{prop}

\begin{proof}
The argument follows closely the proof of Michael Stoll given in \cite[Proposition~8.5]{RN487}. We recall only the main steps. Let  
	\[
	D = \sum_{i=1}^m P_i - m\cdot \infty
	\]
	be a divisor of degree zero, which we may suppose satisfies the following hypotheses: $x(P_i)$ is not a root of $f_0(x)$(see \cite[Lemma 2.2]{RN8}); the degree of $D^+$, namely $m$, is at most $g=\frac{(d-1)(p-1)}{2}$ by Riemann--Roch; the values $x(P_i)$ are pairwise distinct; and each point $P_i$ has integral coordinates (otherwise the claim follows from Lemma~\ref{lemma:negvaluation}).
	
	To simplify notation, write $P_i = (a_i, b_i)$. Then
	\[
	\delta_v(D) = \prod_{i=1}^m (a_i - T_v) = (-1)^m q(T_v),
	\]
	where
	\[
	q(x) = (x-a_1)\cdots (x-a_m) \in \mathcal O_v[x].
	\]
	There exists a unique polynomial $R(x) \in K_v[x]$ of degree at most $m-1$ such that $R(a_i) = b_i$.  
	Choose a uniformizer $\varpi$ of $K_v$, and let $t \ge 0$ be the smallest integer such that
	\[
	S(x) := \varpi^t R(x) \in \mathcal O_v[x].
	\]
	Let $I_D \subset \mathcal O_{L_v}$ be the $\mathcal O_v[T_v]$-ideal generated by $q(T_v)$ and $S(T_v)$. From \cite[Proposition~8.5]{RN487}, we have
	\[
	N(I_D) \equiv \varpi^{mt} \prod_{i=1}^m b_i 
	\text{ and } 
	N(q(T_v)) = \prod_{i=1}^m b_i^{\,p}.
	\]
	hence
	\[
	N\bigl(I_D(\varpi^{mt/d})^{-1}\bigr)^p = N(q(T_v)),
	\]
	where $d = \deg(f_0)$. Set
	$
	J_D := I_D(\varpi^{mt/d})^{-1},
	$
	which is again an $\mathcal O_v[T_v]$-ideal. We have $J_D^p = (q(T_v)).$
	Consequently,
	\[
	\mathcal O_v[T_v] \subset \operatorname{End}(J_D) \subset \operatorname{End}(J_D^p)
	\subset \prod_{i=1}^{r_v} \mathcal O_v[T_{v,i}].
	\]
	The ring $\mathcal O_v[T_v]$ is generated over $\mathcal O_v$ by a single element, so it is Gorenstein of dimension one. In particular, an $\mathcal O_v[T_v]$-ideal is principal if and only if it is proper; hence $J_D$ is a principal $\mathcal O_v[T_v]$-ideal. Write
	$
	J_D = (\alpha)
	$
	for some $\alpha \in \mathcal O_{L_v}$.
	From the identity $J_D^p = (q(T_v))$, we obtain $q(T_v)$ is a $p$-th power up to multiplication by a unit. Recalling that
	$
	\delta_v(D) = \prod_{i=1}^m (a_i - T_v) = (-1)^m q(T_v),
	$
	we conclude that
	$
	\delta_v(D) \in \HOv.
	$
\end{proof}

We next compute the dimensions of the local Kummer image and the relevant norm-one groups.

\begin{prop}\label{dimension}
	For every finite place \(v\) of \(K\), one has
	\begin{enumerate}
		\item
		$\dim_{\mathbb F_p} \Im(\delta_v)
		=
		r_v-1+g d_v.
		$
		\item $\dim_{\mathbb F_p}
		\left(L_v^\times/L_v^{\times p}\right)_{\N\equiv 1}
		=
		2(r_v-1)+2g d_v.$ 
		
		\item $\dim_{\mathbb F_p}
		\left(
		\mathcal O_{L_v}^{\times}/\mathcal O_{L_v}^{\times p}
		\right)_{\N\equiv 1}
		=
		r_v-1+2g d_v.$
	\end{enumerate}	
\end{prop}

\begin{proof}
	 (1) follows from \cite[Corollary~3.6]{RN31}, and Lemma~\ref{dimTate} then gives (2) immediately.	
	
	It remains to prove \textup{(3)}. Write
	$
	L_v=\prod_{i=1}^{r_v}L_{v,i}.
	$
	There exits a short exact sequence
	\[
	0
	\longrightarrow
	\left(
	\mathcal O_{L_v}^{\times}/
	\mathcal O_{L_v}^{\times p}
	\right)_{\N\equiv 1}
	\longrightarrow
	\mathcal O_{L_v}^{\times}/
	\mathcal O_{L_v}^{\times p}
	\xrightarrow{\;\N\;}
	\mathcal O_{K_v}^{\times}/
	\mathcal O_{K_v}^{\times p}
	\longrightarrow 0,
	\]
	where the norm map $\N$ is surjective because $p\nmid d$. Then assertion (4) follows by
	applying \cite[Chapter~II, Proposition~5.7]{Neukirch1999}.

\end{proof}

\subsection{A natural contribution to the local Kummer images}
\label{subsec:natural-local-contribution}
The following analysis is analogous to the results of Barrera Salazar et al. \cite{RN27,RN34}.
Throughout this subsection, let \(v\mid p\) and \(\pi\) be a uniformizer of
\(\BQ_p(\zeta_p)\). Let \(\widehat J\)
denote the formal group of \(J\) at the origin. Evaluating the power-series parametrisation produces explicit classes
in \(J(K_v)/\Pi J(K_v)\). Their images under \(\delta_v\) define a natural subspace of the local Kummer image.

	Recall that the superelliptic curve $C: y^p = f(x)$, where
	$
	f(x) = x^{d^\prime} + a_1 x^{d^\prime-1} + \cdots + a_{d^\prime}.
	$
	To analyze the local behavior near infinity, we introduce new variables
	\[
	z = \frac{y}{x^{\frac{d^\prime+1}{p}}}
	\quad \text{and} \quad
	w = \frac{1}{x},
	\]
	With respect to these coordinates, the origin $O$ on the Jacobian corresponds
	to the point $(z,w) = (0,0)$, and $z$ is a local uniformizer at $O$.
	
	Substituting into the defining equation of $C$, we obtain
	$
	w = z^p - a_1 w^2 - \cdots - a_{d^\prime} w^{d^\prime+1}.
	$
	Iterating this relation allows us to express $w$ as a power series in $z$:
	$
	w = z^p - a_1 z^{2p} + (2a_1^2 - a_2) z^{3p} + \cdots .
	$
	Consequently, we derive Laurent expansions
	\[
	x(z)=z^{-p}+z^{p-1}A(z),
	\qquad
	y(z)=z^{-d^\prime}+z^{2p-d^\prime-1}B(z),
	\]
	where
	$
	A(z),B(z)\in \mathcal O_v[[z]].
	$
	For every \(w\in\mathcal O_v\), the element \(z=\pi w\) belongs to the
	maximal ideal of \(\mathcal O_v\). Consequently, the above power series converge in
	\(\mathcal O_v\).
	Thus we obtain a point
	\[
	P_w=(x(\pi w),y(\pi w))\in C(K_v).
	\]
	We compute the descent image of \(P_w-\infty\):
	\[
	\begin{aligned}
		\delta_v(P_w-\infty)
		&=x(\pi w)-T_v
		=(\pi w)^{-p}
		\left(
		1-\pi^pT_vw^p+\pi^{2p-1}w^pA(\pi w)
		\right) \\
		&=(\pi w)^{-p}
		(1-\pi^pT_vw^p)
		\left(
		1+\frac{\pi^{2p-1}w^pA(\pi w)}
		{1-\pi^pT_vw^p}
		\right).
	\end{aligned}
	\]
	By the local $p$-th power criterion (cf. Lemma \ref{lem:local-pth-power}), 
	\[
	\delta_v(P_w-\infty)
	=
	\left[
	1-\pi^pT_vw^p
	\right].
	\]
	Consequently, the following definition provides an explicit subgroup of the local Kummer image $\Im(\delta_v)$. 
\begin{defn}
	Define 
	\[
	W_v
	=
	\left\{
	\left[
	1-\pi^pT_vw^p
	\right]
	|
	w\in \mathcal O_v
	\right\}\subset \HOv.
	\]
\end{defn}
In the following proposition we gives a linear description of the space $W_v$.

\begin{defn}Define

	\[
	\mathcal U_v
	=
	\left\{
	\left[1+\pi^p\beta\right]|
	\beta\in \mathcal O_{L_v}
	\right\}
	\subset \HOv.
	\] 

\end{defn}
It is clear from the definitions that
$
W_v \subseteq \mathcal U_v.
$ Writing
\[
\beta=(\beta_1,\ldots,\beta_{r_v}),
\qquad
\beta_i\in \mathcal O_{L_{v,i}},
\]
we define
\[
\Phi_v:
\mathcal U_v
\longrightarrow
\mathbb F_p^{r_v}
\]
by
\[
\Phi_v(1+\pi^p\beta)
=
\left(
\operatorname{Tr}_{k_{v,1}/\mathbb F_p}(\overline\beta_1),
\ldots,
\operatorname{Tr}_{k_{v,r_v}/\mathbb F_p}(\overline\beta_{r_v})
\right).
\]
Denote
\[
S_v
=
\left\{
(s_1,\ldots,s_{r_v})\in \mathbb F_p^{r_v}
:
\sum_{i=1}^{r_v} e_{v,i}s_i=0
\right\}.
\]

\begin{prop}\label{part:U4cardinality}
	The map \(\Phi_v\) induces an isomorphism
	\[
	\Phi_v:\mathcal U_v\xrightarrow{\sim} S_v.
	\]
	In particular,
	\[
	\dim_{\mathbb F_p}\mathcal U_v=r_v-1.
	\]
\end{prop}
Before giving the proof of the above proposition, we first present three lemmas as follows.

\begin{lem}\label{lem:local-pth-power}
	Let \(K_v\) be a field complete under a discrete valuation \(v\); suppose 
	that \(K_v\) has characteristic zero and its residue field has characteristic 
	\(p \neq 0\). Let \(e = v(p)\) be the absolute ramification index of \(K_v\). 
	Let  \(\mathcal{O}_v\) denote 
	its ring of integers. If \(t \in \mathcal{O}_v\) satisfies
	\[
	v(t) > \frac{ep}{p-1},
	\]
	then there exists \(\alpha \in \mathcal{O}_v^\times\) such that
	\[
	1 + t = \alpha^p.
	\]
\end{lem}
\begin{proof}
	This result is a direct consequence of Serre's description of the p-th power map on higher unit groups. (cf. \cite[Chapter XIV, Section 4, Proposition 9]{SerreLocalFields}). 
\end{proof}

\begin{lem}\label{Trace condition}
	Let $v$ be a place of $K$ lying above $p$, and let $K_v$ be the corresponding local field. 
	Denote by $\mathcal{O}_v$ the ring of integers of $K_v$, and by $k_v$ the residue field of $K_v$. 
	Let $\pi$ be a uniformizer of $\mathbb{Q}_p(\zeta_p)$. 
	For $\alpha \in \mathcal{O}_v$, let $\overline{\alpha}$ denote its image in $k_v$. 
	Then
	\[
	1 + \pi^p \alpha \in \mathcal{O}_v^{\times p}
	\iff
	\operatorname{Tr}_{k_v/\mathbb{F}_p}(\overline{\alpha}) = 0.
	\]
\end{lem}

\begin{proof}
	Suppose that
	\[
	1 + \pi^p \alpha = \beta^p
	\]
	for some $\beta \in \mathcal O_v^\times$. By Lemma~\ref{lem:local-pth-power}, we have $\beta \equiv 1 \pmod{\pi}$. Let $\mathfrak p$ denote the maximal ideal of $\mathcal O_v$. We claim that $1 + \pi^p \alpha$ is a $p$-th power in $\mathcal O_v^\times$ if and only if there exists $\gamma \in \mathcal O_v$ such that
	\[
	\alpha \equiv -\gamma + \gamma^p \pmod{\mathfrak p}.
	\]
	
	To see this, write $\beta = 1 + \pi \gamma$ with $\gamma \in \mathcal O_v$, and consider the expansion
	\[
	(1 + \pi \gamma)^p
	=
	1 + p \pi \gamma + \sum_{i=2}^{p-1} \binom{p}{i} (\pi \gamma)^i + \pi^p \gamma^p.
	\]
	Since $p = -u \pi^{p-1}$ for some unit $u \in \mathcal O_v^\times$, we may rewrite this as
	\[
	(1 + \pi \gamma)^p
	=
	1 + \pi^p
	\left(
	-u\gamma
	+ \sum_{i=2}^{p-1} \binom{p}{i} \frac{(\pi \gamma)^i}{\pi^p}
	+ \gamma^p
	\right).
	\]
	Using the identity $\prod_{i=1}^{p-1}(1-\zeta_p^i)=p$, we write
	\[
	u = -\prod_{i=1}^{p-1}(1+\zeta_p+\cdots+\zeta_p^{i-1}) \in \mathcal O_v^\times.
	\]
	Since $\zeta_p \equiv 1 \pmod{\pi}$, we have $u \equiv -(p-1)! \equiv 1 \pmod{\pi}$ by Wilson's theorem, which shows that
	\[
	\alpha \equiv -\gamma + \gamma^p \pmod{\mathfrak p}.
	\]
	In $k_v$, we have
	\[
	\overline{\alpha} = \overline{\gamma}^p - \overline{\gamma}.
	\]
	Thus
	\[
	\overline{\alpha} \in \Im(F-1),
	\]
	where
	\[
	F : k_v \to k_v, \qquad F(x) = x^p.
	\]
	The sufficiency is immediate by Lemma~\ref{lem:local-pth-power}.
	
	By additive Hilbert 90,
	\[
	H^1(\operatorname{Gal}(k/\mathbb F_p), k_v) = 0.
	\]
	We have the exact sequence
	\[
	0 \longrightarrow \mathbb F_p
	\longrightarrow k
	\xrightarrow{F-1}
	k
	\xrightarrow{\operatorname{Tr}_{k/\mathbb F_p}}
	\mathbb F_p
	\longrightarrow 0.
	\]
	Hence
	\[
	\Im(F-1) = \Ker \operatorname{Tr}_{k/\mathbb F_p}.
	\]
	Therefore
	\[
	\operatorname{Tr}_{k/\mathbb F_p}(\overline{\alpha}) = 0.
	\]
	
	Conversely, suppose that
	\[
	\operatorname{Tr}_{k/\mathbb F_p}(\overline{\alpha}) = 0.
	\]
	Then
	\[
	\overline{\alpha} \in \Ker \operatorname{Tr}_{k/\mathbb F_p} = \Im(F-1).
	\]
	Thus there exists $\overline{\gamma} \in k_v$ such that
	\[
	\overline{\alpha} = \overline{\gamma}^p - \overline{\gamma}.
	\]
	Choose a lift $\gamma \in \mathcal O_v$ of $\overline{\gamma}$. Then, by the same computation,
	\[
	(1+\pi\gamma)^p
	\equiv
	1+\pi^p(\gamma^p-\gamma)
	\pmod{\pi^p \mathfrak p}.
	\]
	Since
	\[
	\alpha - (\gamma^p-\gamma) \in \mathfrak p,
	\]
	we obtain
	\[
	1+\pi^p\alpha \in \mathcal O_v^{\times p}.
	\]
\end{proof}

\begin{lem}\label{tracezero}
	For \(1\leq i\leq r_v\), let
	\(L_{v,i}\) be a factor of \(L_v\), and let
	\(e_{v,i}=e(L_{v,i}/K_v)\) denote the ramification index of
	\(L_{v,i}/K_v\). Let \(k_{v,i}\) and \(k_v\) denote the residue
	fields of \(L_{v,i}\) and \(K_v\), respectively. Write
	\(\overline{T}_{v,i}\) for the reduction of \(T_{v,i}\).
	One has
	\[
	\sum_{i=1}^{r_v}
	e_{v,i}
	\operatorname{Tr}_{k_{v,i}/k_v}(\overline {T_{v,i}})
	=
	0
	\quad
	\text{in } k_v.
	\]
\end{lem}

\begin{proof}
	The coefficient of \(x^{d-1}\) in \(f_0(x)\) is, up to sign, the trace
	\[
	\operatorname{Tr}_{L_v/K_v}(T_v)
	=
	\sum_{i=1}^{r_v}
	\operatorname{Tr}_{L_{v,i}/K_v}(T_{v,i}).
	\]
	By Hypothesis~\ref{hyp:local}, this trace lies in \(\mathfrak p_v\). Reducing modulo
	\(\mathfrak p_v\), and using the standard congruence
	\[
	\operatorname{Tr}_{L_{v,i}/K_v}(T_{v,i})
	\equiv
	e_{v,i}
	\operatorname{Tr}_{k_{v,i}/k_v}(\overline {T_{v,i}})
	\pmod{\mathfrak p_v},
	\]
	we obtain the desired relation. 
\end{proof}

\begin{proof}[Proof of Proposition~\ref{part:U4cardinality}]
	We first verify that $\Phi_v$ is a group homomorphism. Let $\alpha, \beta \in \mathcal O_{L_v}$. Then
	\[
	(1+\pi^p\alpha)(1+\pi^p\beta)
	=
	1+\pi^p(\alpha+\beta)+\pi^{2p}\alpha\beta
	=
	(1+\pi^p(\alpha+\beta))(1+\pi^{2p}z)
	\]
	where
	\[
	z = \frac{\alpha\beta}{1+\pi^p(\alpha+\beta)} \in \mathcal O_{L_v}.
	\]
	Indeed, $1+\pi^p(\alpha+\beta)$ is a unit in $\mathcal O_{L_v}$. By Lemma~\ref{lem:local-pth-power}, the element $1+\pi^{2p}z$ is a $p$-th power in $\mathcal O_{L_v}^{\times}$. Hence
	\[
	(1+\pi^p\alpha)(1+\pi^p\beta)
	\equiv
	1+\pi^p(\alpha+\beta)
	\pmod{\mathcal O_{L_v}^{\times p}}.
	\]
	It follows that $\Phi_v$ is a homomorphism.
	
	Next we prove that $\Phi_v$ is well-defined on classes modulo $\mathcal O_{L_v}^{\times p}$. Suppose that
	\[
	1+\pi^p\beta \in \mathcal O_{L_v}^{\times p}.
	\]
	Writing
	\[
	\beta = (\beta_1,\ldots,\beta_{r_v}), \qquad \beta_i \in \mathcal O_{v,i},
	\]
	we obtain, for every $i = 1,\ldots,r_v$,
	$
	1+\pi^p\beta_i \in \mathcal O_{v,i}^{\times p}.
	$
	By Lemma~\ref{Trace condition}, this is equivalent to
	\[
	\operatorname{Tr}_{k_{v,i}/\mathbb F_p}(\overline{\beta}_i) = 0.
	\]
	Therefore
	$\Phi_v$ is well-defined on $\mathcal U_v$.
	
	We now show that $\Phi_v$ is injective. If
	$
	\Phi_v(1+\pi^p\beta) = 0,
	$
	then
	$
	\operatorname{Tr}_{k_{v,i}/\mathbb F_p}(\overline{\beta}_i) = 0
	\quad \text{for all } i = 1,\ldots,r_v.
	$
	Again by Lemma~\ref{Trace condition}, each component $1+\pi^p\beta_i$ is a $p$-th power in $\mathcal O_{v,i}^{\times}$. Hence
	\[
	1+\pi^p\beta \in \mathcal O_{L_v}^{\times p}.
	\]
	Thus the class of $1+\pi^p\beta$ is trivial in $\mathcal U_v$, so $\Phi_v$ is injective.
	
	It remains to determine the image. Let
	$
	\beta = (\beta_1,\ldots,\beta_{r_v}) \in \mathcal O_{L_v}.
	$
	The norm condition defining $\mathcal U_v$ says that
	\[
	N_v(1+\pi^p\beta) \in K_v^{\times p}.
	\]
	Expanding the norm componentwise, we have
	\[
	N_{L_{v,i}/K_v}(1+\pi^p\beta_i)
	\equiv
	1+\pi^p \operatorname{Tr}_{L_{v,i}/K_v}(\beta_i)
	\pmod{\pi^p \mathfrak p_v}.
	\]
	Reducing the trace modulo $\mathfrak p_v$, one obtains
	\[
	\operatorname{Tr}_{L_{v,i}/K_v}(\beta_i)
	\equiv
	e_{v,i} \operatorname{Tr}_{k_{v,i}/k_v}(\overline{\beta}_i)
	\pmod{\mathfrak p_v}.
	\]
	Therefore
	\[
	N(1+\pi^p\beta)
	\equiv
	1+\pi^p
	\sum_{i=1}^{r_v}
	e_{v,i} \operatorname{Tr}_{k_{v,i}/k_v}(\overline{\beta}_i)
	\pmod{\pi^p \mathfrak p_v}.
	\]
	By Lemma~\ref{Trace condition}, this norm is a $p$-th power in $K_v^\times$ if and only if
	\[
	\sum_{i=1}^{r_v}
	e_{v,i}
	\operatorname{Tr}_{k_{v,i}/\mathbb F_p}(\overline{\beta}_i)
	=
	0.
	\]
	In other words,
	$
	\Phi_v(\mathcal U_v) \subset S_v.
	$
	
	Conversely, let
	$
	(s_1,\ldots,s_{r_v}) \in S_v.
	$
	Since the trace maps
	$
	\operatorname{Tr}_{k_{v,i}/\mathbb F_p} : k_{v,i} \longrightarrow \mathbb F_p
	$
	are nonzero $\mathbb F_p$-linear maps, they are surjective. Hence we may choose $\overline{\beta}_i \in k_{v,i}$ such that
	\[
	\operatorname{Tr}_{k_{v,i}/\mathbb F_p}(\overline{\beta}_i) = s_i.
	\]
	Choose lifts $\beta_i \in \mathcal O_{v,i}$, and put
	$
	\beta = (\beta_1,\ldots,\beta_{r_v}) \in \mathcal O_{L_v}.
	$
	Because $(s_1,\ldots,s_{r_v}) \in S_v$, we have
	\[
	\sum_{i=1}^{r_v} e_{v,i} s_i = 0.
	\]
	Thus the preceding norm computation shows that
	\[
	N(1+\pi^p\beta) \in K_v^{\times p}.
	\]
	Hence $1+\pi^p\beta$ represents an element of $\mathcal U_v$, and by construction
	\[
	\Phi_v(1+\pi^p\beta) = (s_1,\ldots,s_{r_v}).
	\]
	Therefore $S_v\subset \Phi_v(\mathcal U_v)  $.
	
	We have proved that $\Phi_v$ is a well-defined bijective homomorphism. Thus
	\[
	\mathcal U_v \simeq S_v.
	\]
	Since
	$
	\dim_{\mathbb F_p} S_v = r_v - 1.
	$
	Then,
	\[
	\dim_{\mathbb F_p} \mathcal U_v = r_v - 1.
	\]
\end{proof}

\begin{defn}\label{Vv}
	For \(1\leq i\leq r_v\), let \(k_{v,i}\) denote the residue field of
	\(L_{v,i}\), and let \(k_v\) denote the residue field of \(K_v\).
	Write \(\overline{{T}_{v,i}}\) for the reduction of \(T_{v,i}\).
	Define the trace space
	\[
	V_v
	=
	\left\langle
	\operatorname{Tr}_{k_{v,i}/k_v}(\overline {T_{v,i}})
	:
	i=1,\ldots,r_v
	\right\rangle_{\mathbb F_p}
	\subset k_v.
	\]
\end{defn}
\begin{thm}\label{Wv}
	One has
	\[
	\dim_{\mathbb F_p}W_v
	=
	\dim_{\mathbb F_p}V_v.
	\]
	Consequently,
	\[
	[\mathcal U_v:W_v]
	=
	p^{r_v-1-\dim_{\mathbb F_p}V_v}.
	\]
\end{thm}
\begin{proof}
	Recall that \(W_v\subset\Im(\delta_v)\). Under the isomorphism
	$
	\Phi_v:\mathcal U_v\xrightarrow{\sim} S_v,
	$
	we have
	\[
	\begin{aligned}
		\Phi_v(1-\pi^pT_vw^p)
		=
		\bigl(
		-\operatorname{Tr}_{k_{v,1}/\mathbb F_p}
		(\overline {T_{v,1}}\overline w^{\,p}),
		\ldots,
		-\operatorname{Tr}_{k_{v,r_v}/\mathbb F_p}
		(\overline {T_{v,r_v}}\overline w^{\,p})
		\bigr).
	\end{aligned}
	\]
	Since the Frobenius map is an automorphism of \(k_v\), the set of
	\(\overline w^{\,p}\) as \(w\) varies through \(\mathcal O_v\) is the whole
	field \(k_v\). Therefore \(\Phi_v(W_v)\) is the image of the linear map
	\[
	k_v
	\longrightarrow
	\mathbb F_p^{r_v},
	\qquad
	u\longmapsto
	\left(
	\operatorname{Tr}_{k_v/\mathbb F_p}
	\bigl(
	\operatorname{Tr}_{k_{v,1}/k_v}(\overline{ T_{v,1}})u
	\bigr),
	\ldots,
	\operatorname{Tr}_{k_v/\mathbb F_p}
	\bigl(
	\operatorname{Tr}_{k_{v,r_v}/k_v}(\overline {T_{v,r_v}})u
	\bigr)
	\right).
	\]
	The trace pairing
	$
	k_v\times k_v\longrightarrow \mathbb F_p,
	(a,u)\longmapsto \operatorname{Tr}_{k_v/\mathbb F_p}(au)
	$
	is perfect. Hence the dimension of the image is exactly the dimension of the
	\(\mathbb F_p\)-span of
	$
	\operatorname{Tr}_{k_{v,i}/k_v}(\overline {T_{v,i}}),
	i=1,\ldots,r_v.
	$
	By definition, this dimension is
	$
	\dim_{\mathbb F_p}V_v.
	$
	Therefore
	$
	\dim_{\mathbb F_p}W_v=\dim_{\mathbb F_p}V_v.
	$
	Since
	$
	\dim_{\mathbb F_p}\mathcal U_v=r_v-1,
	$
	we obtain
	$
	[\mathcal U_v:W_v]
	=
	p^{r_v-1-\dim_{\mathbb F_p}V_v}.
	$
\end{proof}

\section{The \texorpdfstring{\(\Pi\)}{Pi}-Selmer group and class group estimates}
\label{sec:selmer-class-group}

\subsection{$\Pi$-Selmer Group}
The short exact sequence of \(G_K\)-modules
\[
0 \longrightarrow J[\Pi] \longrightarrow J(\overline K) \xrightarrow{\Pi} J(\overline K) \longrightarrow 0
\]
gives rise to the global Kummer exact sequence
\[
0 \longrightarrow J(K)/\Pi J(K) \xrightarrow{\kappa_{K,\Pi}} H^1(K,J[\Pi]) \longrightarrow H^1(K,J)[\Pi].
\]
For every place \(v\) of \(K\), the same construction over \(K_v\) yields
\[
0 \longrightarrow J(K_v)/\Pi J(K_v) \xrightarrow{\kappa_{K_v,\Pi}} H^1(K_v,J[\Pi]) \longrightarrow H^1(K_v,J)[\Pi].
\]
For each place \(v\) of \(K\), we denote by \(\operatorname{loc}_v\) the localization map
$
H^1(K,J[\Pi]) \longrightarrow H^1(K_v,J[\Pi]).
$
For a class \([\alpha] \in \left(L^\times/L^{\times p}\right)_{N\equiv 1}\), we write \([\alpha]_v\) for its image in
$
\left(L_v^\times/L_v^{\times p}\right)_{N\equiv 1}.
$
\begin{defn}
	The \(\Pi\)-Selmer group of \(J\) over \(K\) is
	$$
	\operatorname{Sel}_{\Pi}(J/K)
	=
	\left\{
	\xi\in H^1(K,J[\Pi])
	:
	\operatorname{loc}_v(\xi)\in \Im(\delta_v)
	, \forall v\in \Sigma
	\right\}.
	$$
\end{defn}
Using the isomorphism in Proposition~\ref{prop:cohom}(3), we may regard the Selmer group as a subgroup
of
$
\left(L^\times/L^{\times p}\right)_{N\equiv 1}.
$
More precisely,
$$
\operatorname{Sel}_{\Pi}(J/K)
=
\left\{
[\alpha]\in
\left(L^\times/L^{\times p}\right)_{N\equiv 1}
:
[\alpha]_v\in \Im(\delta_v)
,\forall v\in \Sigma
\right\}.
$$
Thus the global Selmer group is obtained by imposing the local Kummer
conditions computed in Section~\ref{sec:local-kummer-images}. 
The local comparisons in Theorem~\ref{localcomputation} motivate the auxiliary Selmer structures
introduced below.

Recall that \(K\) is a number field containing \(\zeta_p\), \(C/K\) is the
superelliptic curve
$
C:y^p=f(x),
$
and \(J=\operatorname{Jac}(C)\). Put
$
L=K[x]/(f_0).
$
Writing \(f_0=f_1\cdots f_r\) as a product of monic irreducible
polynomials over \(K\), we have
$
L\simeq\prod_{i=1}^r L_i,
L_i=K[x]/(f_i),
$
and set
$
\operatorname{Cl}(L)
:=\prod_{i=1}^r\operatorname{Cl}(L_i).
$
We use the analogous product notation for ideals, rings of integers,
and localizations. For
\(\alpha=(\alpha_i)_{i=1}^r\in L^\times\), put
$
L(\sqrt[p]{\alpha})
:=\prod_{i=1}^rL_i(\sqrt[p]{\alpha_i}).
$
We say that \(L(\sqrt[p]{\alpha})/L\) is unramified above a place \(v\)
of \(K\) if each extension
\(L_i(\sqrt[p]{\alpha_i})/L_i\) is unramified at every place of \(L_i\)
above \(v\).

\subsection{Auxiliary Selmer structures}

We now introduce three Selmer structures on \(J[\Pi]\).

For a finite place \(v\), define the unramified local condition by
\[
\mathcal F_{\mathrm{ur}}(v)
:=
\left\{
\xi_v\in (L_v^\times/L_v^{\times p})_{\N\equiv1}:
L_v(\sqrt[p]{\xi_v})/L_v
\text{ is unramified}
\right\}.
\]
At archimedean places we take the full local condition:
\[
\mathcal F_{\mathrm{ur}}(v)
:=
H^1(K_v,J[\Pi]).
\]
The associated Selmer group is
\[
H^1_{\mathcal F_{\mathrm{ur}}}(K,J[\Pi])
:=
\left\{
\xi\in H^1(K,J[\Pi]):
\operatorname{loc}_v(\xi)\in \mathcal F_{\mathrm{ur}}(v)
, \forall v \in \Sigma
\right\}.
\]

Next define the unit local condition.  For a finite place \(v\), put
\[
\mathcal F_{\mathrm{unit}}(v)
:=
\left(
\mathcal O_{L_v}^{\times}/
\mathcal O_{L_v}^{\times p}
\right)_{\N\equiv1},
\]
and at archimedean places again take the full local condition.  The
corresponding Selmer group is
\[
H^1_{\mathcal F_{\mathrm{unit}}}(K,J[\Pi])
:=
\left\{
\xi\in H^1(K,J[\Pi]):
\operatorname{loc}_v(\xi)\in\mathcal F_{\mathrm{unit}}(v)
, \forall v \in \Sigma
\right\}.
\]

Finally, for each finite place \(v\mid p\), let
\[
W_v\subset \CU_{v}\subset \mathcal O_{L_v}^{\times}
\]
be the local subgroup defined in Section~\ref{sec:local-kummer-images}.  Define a third Selmer
structure \(\mathcal F_W\) by
\[
\mathcal F_W(v)
:=
\begin{cases}
	\mathcal F_{\mathrm{ur}}(v),
	& v\nmid p,\\[4pt]
	\mathcal F_{\mathrm{ur}}(v) \cap
	\Im\left(
	W_v\longrightarrow (L_v^\times/L_v^{\times p})_{\N\equiv1}
	\right),
	& v\mid p,
\end{cases}
\]
for finite places \(v\), and by the full local condition at archimedean
places.  The corresponding Selmer group is
\[
H^1_{\mathcal F_W}(K,J[\Pi])
:=
\left\{
\xi\in H^1(K,J[\Pi]):
\operatorname{loc}_v(\xi)\in\mathcal F_W(v)
, \forall v \in \Sigma
\right\}.
\]

By construction one has inclusions of Selmer structures
\[
\mathcal F_W\subset
\mathcal F_{\mathrm{ur}}\subset
\mathcal F_{\mathrm{unit}},
\]
and hence
\[
H^1_{\mathcal F_W}(K,J[\Pi])
\subset
H^1_{\mathcal F_{\mathrm{ur}}}(K,J[\Pi])
\subset
H^1_{\mathcal F_{\mathrm{unit}}}(K,J[\Pi]).
\]
\begin{lem}
	The set $H^1_{\mathcal F_{\mathrm{ur}}}(K,J[\Pi])$ equals the set of elements $[\alpha] \in \HK$ satisfying the following local conditions
	\begin{enumerate}
		\item for finite places $v\nmid p$, the v-adic valuation of $\alpha$ is divisible by $p$.
		\item for finite places $v\mid p$, $\alpha = 1+\pi^p\CO_{L_{v}}$
	\end{enumerate}
\end{lem}
\begin{proof}
	This is a standard local form of Kummer theory; it appears in
	the book of Henri Cohen and Georges Gras, see \cite[Chapter~I, Theorem~6.3]{RN45}.	
\end{proof}

\begin{prop}\label{prop:selmer-sandwich}
	Assume Hypothesis~\ref{hyp:local}.  Then
	\[
	H^1_{\mathcal F_W}(K,J[\Pi])
	\subset
	\operatorname{Sel}_{\Pi}(J)
	\subset
	H^1_{\mathcal F_{\mathrm{unit}}}(K,J[\Pi]).
	\]
\end{prop}

\begin{proof}
	Let \(v\) be a finite place of \(K\).  If \(v\nmid p\), the local
	description of the Kummer image from Section~3 shows that the unramified
	local condition is contained in \(\Im(\delta_v)\).  If
	\(v\mid p\), Theorem~\ref{Wv} gives
	\[
	\Im\left(
	W_v\to L_v^\times/L_v^{\times p}
	\right)
	\subset
	\Im(\delta_v).
	\]
	Thus every class in
	\(H^1_{\mathcal F_W}(K,J[\Pi])\) satisfies the local Selmer condition at
	every finite place.
	
	By Proposition~\ref{thm:integrality} and Hypothesis~\ref{hyp:local}, one has
	\[
	\Im(\delta_v)
	\subset
	\left(
	\mathcal O_{L_v}^{\times}/
	\mathcal O_{L_v}^{\times p}
	\right)_{\N\equiv1}
	=
	\mathcal F_{\mathrm{unit}}(v)
	\]
	for every finite place \(v\).  Therefore every element of
	\(\operatorname{Sel}_{\Pi}(J)\) lies in
	\(H^1_{\mathcal F_{\mathrm{unit}}}(K,J[\Pi])\).  The archimedean places
	impose no additional condition in the present comparison.
\end{proof}

Let
\[
\operatorname{Cl}(L)/p\operatorname{Cl}(L)_{\N\equiv0}
:=
\Ker\left(
\N_{L/K}:
\operatorname{Cl}(L)/p\operatorname{Cl}(L)
\longrightarrow
\operatorname{Cl}(K)/p\operatorname{Cl}(K)
\right),
\]
\[
\operatorname{Cl}(L)[p]_{\N\equiv0}
:=
\Ker\left(
\N_{L/K}:\operatorname{Cl}(L)[p]
\longrightarrow
\operatorname{Cl}(K)[p]
\right).
\]

\begin{prop}\label{prop:unramified-classgroup}
	There is a canonical perfect pairing
	\[
	H^1_{\mathcal F_{\mathrm{ur}}}(K,J[\Pi])
	\times
	\left(
	\operatorname{Cl}(L)/p\operatorname{Cl}(L)
	\right)_{\N\equiv 0}
	\longrightarrow
	\mu_p.
	\]
	In particular,
	\[
	\dim_{\mathbb F_p}
	H^1_{\mathcal F_{\mathrm{ur}}}(K,J[\Pi])
	=
	\dim_{\mathbb F_p}
	\operatorname{Cl}(L)[p]_{\N\equiv 0}.
	\]
\end{prop}

\begin{proof}
	Let \(H_L\) be the Hilbert class field of the étale \(K\)-algebra \(L\),
	that is, the product of the Hilbert class fields of the fields \(L_i\).
	Class field theory gives
	$
	\operatorname{Gal}(H_L/L)
	\simeq
	\operatorname{Cl}(L).
	$
	Consequently,
	\[
	H^1_{\mathrm{ur}}(L,\mu_p)
	=
	\Hom\!\left(
	\operatorname{Gal}(H_L/L),\mu_p
	\right)
	\simeq
	\Hom\!\left(
	\operatorname{Cl}(L)/p\operatorname{Cl}(L),
	\mu_p
	\right).
	\]
	Here \(H^1_{\mathrm{ur}}(L,\mu_p)\) denotes the subgroup consisting
	of classes whose corresponding Kummer extensions are unramified at
	every finite place.
	
	The exact sequence
	\[
	0
	\longrightarrow J[\Pi]
	\longrightarrow \Res_{L/K}\mu_p
	\xrightarrow{\N_{L/K}}
	\mu_p
	\longrightarrow 0,
	\]
	together with Shapiro's lemma, identifies the unramified Selmer group
	with
	\[
	H^1_{\mathcal F_{\mathrm{ur}}}(K,J[\Pi])
	\simeq
	\Ker\!\left(
	H^1_{\mathrm{ur}}(L,\mu_p)
	\xrightarrow{\operatorname{Cor}_{L/K}}
	H^1_{\mathrm{ur}}(K,\mu_p)
	\right).
	\]
	
	Let
	$
	\iota_{L/K}\colon
	\operatorname{Cl}(K)
	\longrightarrow
	\operatorname{Cl}(L)
	$
	denote extension of ideals. The functoriality of the Artin
	reciprocity map shows that corestriction is adjoint to extension of
	ideals. More precisely, for
	\(\chi\in H^1_{\mathrm{ur}}(L,\mu_p)\) and
	\(c\in\operatorname{Cl}(K)\), one has
	$
	\bigl\langle
	\operatorname{Cor}_{L/K}\chi,c
	\bigr\rangle_K
	=
	\bigl\langle
	\chi,\iota_{L/K}(c)
	\bigr\rangle_L.
	$
	It follows that
	\[
	H^1_{\mathcal F_{\mathrm{ur}}}(K,J[\Pi])
	\simeq
	\Hom\!\left(
	\frac{\operatorname{Cl}(L)/p\operatorname{Cl}(L)}
	{\iota_{L/K}
		\bigl(\operatorname{Cl}(K)/p\operatorname{Cl}(K)\bigr)},
	\mu_p
	\right).
	\]
    There is a canonical isomorphism
	\[
	\frac{\operatorname{Cl}(L)/p\operatorname{Cl}(L)}
	{\iota_{L/K}
		\bigl(\operatorname{Cl}(K)/p\operatorname{Cl}(K)\bigr)}
	\xrightarrow{\sim}
	\Ker(\N_{L/K}).
	\]
	We thus obtain
	\[
	H^1_{\mathcal F_{\mathrm{ur}}}(K,J[\Pi])
	\simeq
	\Hom\!\left(
	\left(
	\operatorname{Cl}(L)/p\operatorname{Cl}(L)
	\right)_{\N\equiv0},
	\mu_p
	\right).
	\]	
	Finally, since $\operatorname{Cl}(L)$ is a finite abelian group, then
	\[
	\dim_{\mathbb F_p}
	\left(
	\operatorname{Cl}(L)/p\operatorname{Cl}(L)
	\right)_{\N\equiv0}
	=
	\dim_{\mathbb F_p}
	\operatorname{Cl}(L)[p]_{\N\equiv0}.
	\]
\end{proof}

\subsection{Bounds for the \(\Pi\)-Selmer group}

We first compare the unit Selmer structure with the unramified Selmer
structure.

\begin{lem}\label{lem:unit-norm-kernel}
	One has
	\[
	\#
	\left(
	\mathcal O_L^\times/
	\mathcal O_L^{\times p}
	\right)_{\N\equiv1}
	=
	p^{\,g[K:\mathbb Q(\zeta_p)]}.
	\]
\end{lem}

\begin{proof}
	There is a short exact sequence
	\[
	0
	\longrightarrow
	\left(
	\mathcal O_{L}^{\times}/
	\mathcal O_{L}^{\times p}
	\right)_{\N\equiv 1}
	\longrightarrow
	\mathcal O_{L}^{\times}/
	\mathcal O_{L}^{\times p}
	\xrightarrow{\;\N\;}
	\mathcal O_{K}^{\times}/
	\mathcal O_{K}^{\times p}
	\longrightarrow 0,
	\]
	Since $K$ contains $\zeta_p$, it is totally imaginary. By Dirichlet's unit theorem,
	\[
	\dim_{\mathbb F_p} \mathcal O_K^\times / \mathcal O_K^{\times p} = \frac{[K:\mathbb Q]}{2}.
	\]
	Similarly, because $L \simeq L_1 \times \cdots \times L_r$, one obtains
	\[
	\dim_{\mathbb F_p} \mathcal O_L^\times / \mathcal O_L^{\times p}
	=
	\frac{d[K:\mathbb Q]}{2}.
	\]
	Hence
	\[
	\dim_{\mathbb F_p}
	\left(
	\mathcal O_L^\times / \mathcal O_L^{\times p}
	\right)_{\N\equiv 1}
	=
	\frac{(d-1)[K:\mathbb Q]}{2}.
	\]
	Recall that
	\[
	[K:\mathbb Q] = (p-1)[K:\mathbb Q(\zeta_p)]\text{ and } g = \frac{(d-1)(p-1)}{2}
	\]
	the claimed formula follows.
\end{proof}
\begin{prop}\label{prop:unit-ur-index}
	One has
	\[
	\left[
	H^1_{\mathcal F_{\mathrm{unit}}}(K,J[\Pi]):
	H^1_{\mathcal F_{\mathrm{ur}}}(K,J[\Pi])
	\right]
	\le
	p^{\,g[K:\mathbb Q(\zeta_p)]}.
	\]
\end{prop}

\begin{proof}
	Let
	\[
	[\alpha]\in
	H^1_{\mathcal F_{\mathrm{unit}}}(K,J[\Pi])
	\subset
	(L^\times/L^{\times p})_{\N\equiv1}.
	\]
	The unit local condition implies that the valuation of \(\alpha\) at
	every finite prime of \(L\) is divisible by \(p\).  Hence there exists a
	fractional ideal \(\mathfrak a\) of \(L\) such that
	\[
	(\alpha)=\mathfrak a^p.
	\]
	Sending \([\alpha]\) to the ideal class of \(\mathfrak a\) gives a
	well-defined homomorphism
	\[
	H^1_{\mathcal F_{\mathrm{unit}}}(K,J[\Pi])
	\longrightarrow
	\operatorname{Cl}(L)[p]_{\N\equiv0}.
	\]
	Its kernel consists of classes represented by global units.  Therefore
	\[
	\#
	H^1_{\mathcal F_{\mathrm{unit}}}(K,J[\Pi])
	\le
	\#
	\operatorname{Cl}(L)[p]_{\N\equiv0}
	\cdot
	\#
	\left(
	\mathcal O_L^\times/
	\mathcal O_L^{\times p}
	\right)_{\N\equiv1}.
	\]
	Using Proposition~\ref{prop:unramified-classgroup} and
	Lemma~\ref{lem:unit-norm-kernel}, we obtain
	\[
	\left[
	H^1_{\mathcal F_{\mathrm{unit}}}(K,J[\Pi]):
	H^1_{\mathcal F_{\mathrm{ur}}}(K,J[\Pi])
	\right]
	\le
	p^{\,g[K:\mathbb Q(\zeta_p)]}.
	\]
\end{proof}

We next compare the unramified Selmer structure with
\(\mathcal F_W\).  For \(v\mid p\), let \(r_v\) be the number of factors
of \(L_v\) over \(K_v\), and let \(V_v\) be the local
\(\mathbb F_p\)-subspace defined in Section~\ref{sec:local-kummer-images}.

\begin{prop}\label{prop:ur-W-index}
	One has
	\[
	\left[
	H^1_{\mathcal F_{\mathrm{ur}}}(K,J[\Pi]):
	H^1_{\mathcal F_W}(K,J[\Pi])
	\right]
	\le
	p^{
		\sum_{v\mid p}
		\left(
		r_v-1-\dim_{\mathbb F_p}W_v
		\right)
	}.
	\]
\end{prop}

\begin{proof}
	Away from \(p\), the two Selmer structures
	\(\mathcal F_{\mathrm{ur}}\) and \(\mathcal F_W\) have the same local
	condition.  Thus the quotient is controlled only by the local conditions
	at places \(v\mid p\).
	
	For such a place \(v\), the local computation of Section~\ref{sec:local-kummer-images} gives
	\[
	\dim_{\mathbb F_p}
	\frac{
		\mathcal F_{\mathrm{ur}}(v)
	}{
		\mathcal F_W(v)
	}
	\le
	r_v-1-\dim_{\mathbb F_p}W_v.
	\]
	Taking the product over all places \(v\mid p\) gives the desired bound.
\end{proof}

\begin{thm}\label{thm:selmer-classgroup-bound}
	Assume Hypothesis~\ref{hyp:local}.  Then
	\[
	\begin{aligned}
		\dim_{\mathbb F_p}
		\operatorname{Cl}(L)[p]_{\N\equiv0}
		-
		\sum_{v\mid p}
		\left(
		r_v-1-\dim_{\mathbb F_p}W_v
		\right)
		\le
		\dim_{\mathbb F_p}
		\operatorname{Sel}_{\Pi}(J)
		\\
		\le
		\dim_{\mathbb F_p}
		\operatorname{Cl}(L)[p]_{\N\equiv0}
		+
		g[K:\mathbb Q(\zeta_p)].
	\end{aligned}
	\]
\end{thm}

\begin{proof}
	By Proposition~\ref{prop:selmer-sandwich},
	\[
	H^1_{\mathcal F_W}(K,J[\Pi])
	\subset
	\operatorname{Sel}_{\Pi}(J)
	\subset
	H^1_{\mathcal F_{\mathrm{unit}}}(K,J[\Pi]).
	\]
	The upper bound follows from Proposition~\ref{prop:unit-ur-index} and
	Proposition~\ref{prop:unramified-classgroup}:
	\[
	\begin{aligned}
		\dim_{\mathbb F_p}
		\operatorname{Sel}_{\Pi}(J)
		&\le
		\dim_{\mathbb F_p}
		H^1_{\mathcal F_{\mathrm{unit}}}(K,J[\Pi])          \\
		&\le
		\dim_{\mathbb F_p}
		H^1_{\mathcal F_{\mathrm{ur}}}(K,J[\Pi])
		+
		g[K:\mathbb Q(\zeta_p)]                              \\
		&=
		\dim_{\mathbb F_p}
		\operatorname{Cl}(L)[p]_{\N\equiv0}
		+
		g[K:\mathbb Q(\zeta_p)].
	\end{aligned}
	\]
	Similarly, Proposition~\ref{prop:ur-W-index} gives
	\[
	\begin{aligned}
		\dim_{\mathbb F_p}
		\operatorname{Sel}_{\Pi}(J)
		&\ge
		\dim_{\mathbb F_p}
		H^1_{\mathcal F_W}(K,J[\Pi])                       \\
		&\ge
		\dim_{\mathbb F_p}
		H^1_{\mathcal F_{\mathrm{ur}}}(K,J[\Pi])
		-
		\sum_{v\mid p}
		\left(
		r_v-1-\dim_{\mathbb F_p}W_v
		\right)                                             \\
		&=
		\dim_{\mathbb F_p}
		\operatorname{Cl}(L)[p]_{\N\equiv0}
		-
		\sum_{v\mid p}
		\left(
		r_v-1-\dim_{\mathbb F_p}W_v
		\right).
	\end{aligned}
	\]
	This proves the theorem.
\end{proof}

\subsection{Examples}
The following examples have been computed using SageMath\cite{sagemath} and Magma\cite{MR1484478}, all examples below suppose that \( K = \mathbb{Q}(\zeta_3) \).
\begin{ex}[Lower bound attained]
	For $C: y^3 = x^2 - 822$, let $L = K(\sqrt{822})$. A class-group computation gives
	\[
	\dim_{\mathbb F_3} \operatorname{Cl}(L)[3]_{\N\equiv 0} = 1.
	\]
	At the unique place $\pi \mid 3$, we have $r_\pi = 1$ and $W_\pi = 0$, hence
	\[
	r_\pi - 1 - \dim_{\mathbb F_3} W_\pi = 0.
	\]
	Moreover,
	\[
	\dim_{\mathbb F_3} \operatorname{Sel}_{\Pi}(J/K) = 1
	\]
	(cf.~\cite[Example~1]{JhaMajumdarShingavekar}). Thus
	\[
	\dim_{\mathbb F_3} \operatorname{Sel}_{\Pi}(J/K)
	=
	\dim_{\mathbb F_3} \operatorname{Cl}(L)[3]_{\N\equiv 0}
	=
	1,
	\]
	attaining the lower bound in Theorem~\ref{thm:selmer-classgroup-bound}.
\end{ex}

\begin{ex}[Upper bound attained]
	For $C: y^3 = x^2 - 359$, let $L = K(\sqrt{359})$. We have
	\[
	\dim_{\mathbb F_3} \operatorname{Cl}(L)[3]_{\N\equiv 0} = 2,
	\qquad
	r_\pi = 1, \quad W_\pi = 0,
	\]
	so $r_\pi - 1 - \dim_{\mathbb F_3} W_\pi = 0$. Also,
	\[
	\dim_{\mathbb F_3} \operatorname{Sel}_{\Pi}(J/K) = 3
	\]
	(see \cite[Example~2]{JhaMajumdarShingavekar}). Therefore
	\[
	\dim_{\mathbb F_3} \operatorname{Sel}_{\Pi}(J/K)
	=
	\dim_{\mathbb F_3} \operatorname{Cl}(L)[3]_{\N\equiv 0} + g
	=
	3,
	\]
	attaining the upper bound in Theorem~\ref{thm:selmer-classgroup-bound}.
\end{ex}

\begin{ex}
	For $C: y^3 = x(x^4 + x^2 + 2)$, we have
	\[
	\dim_{\mathbb F_3} \operatorname{Cl}(L)[3]_{\N\equiv 0} = 0,
	\qquad
	r_\pi = 2, \quad W_\pi = 0,
	\]
	hence $r_\pi - 1 - \dim_{\mathbb F_3} W_\pi = 1$. A direct SageMath computation gives
	\[
	\dim_{\mathbb F_3} \operatorname{Sel}_{\Pi}(J/K) = 2.
	\]
	Theorem~\ref{thm:selmer-classgroup-bound} yields
	\[
	0 \leq \dim_{\mathbb F_3} \operatorname{Sel}_{\Pi}(J/K) \leq 4.
	\]
\end{ex}

\bibliographystyle{alpha}
\bibliography{reference}
\end{document}